\documentclass[11pt]{article}
\usepackage{amsmath,amssymb,amsthm,color}
\usepackage{comment}
\usepackage{xcolor}

\newtheorem{theorem}{Theorem}[section]
\newtheorem{corollary}{Corollary}[section]
\newtheorem{lemma}{Lemma}[section]
\newtheorem{claim}{Claim}[section]

\newtheorem{remark}{Remark}[section]

\def\R{{\mathfrak R}\, }
\def\D{{\mathfrak D}\, }
\def\A{{\mathfrak A}\, }

\begin{document}
\baselineskip.20in

\begin{center}{\bf \LARGE $*$-Lie-type maps on $C^*$-algebras}\\
\vspace{.2in}
\noindent {\bf Ruth Nascimento Ferreira}\\
{\it Federal University of Technology,\\
Avenida Professora Laura Pacheco Bastos, 800,\\
85053-510, Guarapuava, Brazil.}\\
e-mail: ruthferreira@utfpr.edu.br
\vspace{.2in}
\\
\noindent {\bf Bruno Leonardo Macedo Ferreira}\\
{\it Federal University of Technology,\\
Avenida Professora Laura Pacheco Bastos, 800,\\
85053-510, Guarapuava, Brazil.}\\
e-mail: brunolmfalg@gmail.com
\vspace{.2in}
\\
\noindent {\bf Henrique Guzzo Junior}\\
{\it University of S\~{a}o Paulo,\\
Rua do Mat\~{a}o, 1010,\\
05508-090, S\~{a}o Paulo, Brazil.}\\
e-mail: guzzo@ime.usp.br
\vspace{.2in}
\\
\noindent {\bf Bruno Tadeu Costa}\\
{\it Federal University of Santa Catarina,\\
Rua Jo\~{a}o Pessoa, 2750,\\   
89036-256, Blumenau, Brazil.}\\
e-mail: b.t.costa@ufsc.br

\end{center}
%\begin{color}{red}
\begin{abstract} 
Let $\A$ and $\A'$ be two  
$C^*$-algebras with identities $I_{\A}$ and $I_{\A'}$, respectively, and $P_1$ and $P_2 = I_{\A} - P_1$ nontrivial symmetric
projections in $\A$. In this paper we study the  characterization of multiplicative $*$-Lie-type maps. In particular, 
if $\mathcal{M}$ is a factor von Neumann algebra then every complex scalar multiplication bijective unital multiplicative 
$*$-Lie-type map is $*$-isomorphism.
\end{abstract}
{\bf {\it AMS 2010 Subject Classification:}} 	47B48, 46L05.\\
{\bf {\it Keywords:}} $C^*$-algebra, factor von Neumann algebras, multiplicative $*$-Lie-type maps, $*$-isomorphism\\

\section{Introduction and Preliminaries}
The study of additivity of maps strongly attracts attention of mathematicians, so that the first quite surprising result is due to Martindale who imposed conditions on a ring such that multiplicative bijective maps are all additive \cite{Mart}. Thenceforth, diverse works have been published considering different types of associative and non-associative algebras. Among them we can mention \cite{Fer, Fer1, Fer2, Fer3, Fer4, bruth, FerGur1, chang}. In order to further develop the study of additivity of maps, the researches incorporated two new product into the theory, presented by Bre$\check{s}$ar and Fo$\check{s}$ner in \cite{brefos1, brefos2}, where the definition is as follows: for $X, Y \in \R$, where $\R$ is a $*-$ring, we denote by $X\bullet Y = XY+YX^{*}$ and $[X, Y]_{*} = XY-YX^{*}$ the $*$-Jordan product and the $*$-Lie product, respectively. In \cite{CuiLi}, the authors proved that a map $\Phi$ between two factor von Newmann algebras is a $*$-ring isomorphism if and only if $\Phi([X, Y]_{*}) = [\Phi(X),\Phi(Y)]_{*}$. In \cite{Ferco}, Ferreira and Costa extended these new products and defined two other types of applications, named multiplicative $*$-Lie n-map and multiplicative $*$-Jordan n-map
and proved that, under certain conditions, an application between $C^*$-algebras that is multiplicative $*$-Lie n-map and multiplicative $*$-Jordan n-map is a $*$-ring isomorphism.
In the second paper of this series \cite{Ferco2}, Ferreira and Costa prove when a multiplicative $*$-Jordan n-map is a $*$-ring isomorphism. As a consequence of their main result, 
they provide an application on von Neumann algebras, factor von Neumann algebras and prime algebras. Furthermore, they generalize the Main Theorem in \cite{LiLuFang}.
With this picture in mind, in this article we will discuss when a multiplicative $*$-Lie n-map is a $*$-isomorphism and, just as it was done in \cite{Ferco2}, 
we provide an application on von Neumann algebras, factor von Neumann algebras and prime algebras.

Let us define the following sequence of polynomials, as presented in \cite{Ferco}: 
$$p_{1_*}(x) = x\, \,  \text{and}\, \,  p_{n_*}(x_1, x_2, \ldots , x_n) = \left[p_{(n-1)_*}(x_1, x_2, \ldots , x_{n-1}) , x_n\right]_{*},$$
for all integers $n \geq 2$. Thus, $p_{2_*}(x_1, x_2) = \left[x_1, x_2\right]_{*}, \ p_{3_*} (x_1, x_2, x_3) = \left[\left[x_1, x_2\right]_{*} , x_3\right]_{*}$, etc. 
Note that $p_{2_*}$ is the product introduced by Bre$\check{s}$ar and Fo$\check{s}$ner \cite{brefos1, brefos2}. Then, using the nomenclature introduced in \cite{Ferco} 
we have a new class of maps (not necessarily additive): given two rings $\R$ and $\R'$, $\varphi : \R \longrightarrow \R'$ is a \textit{multiplicative $*$-Lie $n$-map} if
\begin{eqnarray*}\label{ident1}
&&\varphi(p_{n_*} (x_1, x_2, . . . , x_i , . . . ,x_n)) =  p_{n_*} (\varphi(x_1), \varphi(x_2), . . . , \varphi(x_i), . . .,\varphi(x_n)),
\end{eqnarray*}
where $n \geq 2$ is an integer. Multiplicative $*$-Lie $2$-map, $*$-Lie $3$-map 
and $*$-Lie $n$-map are collectively referred to as \textit{multiplicative $*$-Lie-type maps}.

By a $C^*$-algebra we mean a complete normed complex algebra (say $\A$) endowed with a conjugate-linear algebra involution $*$, 
satisfying $||a^*a||=||a||^2$ for all $a \in \A$. Moreover, a $C^*$-algebra is a {\it prime} $C^*$-algebra if $A\A B = 0$ for 
$A,B \in \A$ implies either $A=0$ or $B=0$.

We find it convenient to invoke the noted Gelfand-Naimark theorem that state a $C^*$-algebra $\A$ is $*$-isomorphic to a
$C^*$-subalgebra $\D \subset \mathcal{B}(\mathcal{H})$, where $\mathcal{H}$ is a Hilbert space. So from now on we shall
consider elements of a $C^*$-algebra as operators.

Let be $P_1$ a nontrivial projection in $\A$ and $P_2 = I_{\A} - P_1$ where $I_{\A}$ is the identity of $\A$. Then $\mathfrak{A}$ has a decomposition
$\mathfrak{A}=\mathfrak{A}_{11}\oplus \mathfrak{A}_{12}\oplus
\mathfrak{A}_{21}\oplus \mathfrak{A}_{22},$ where
$\mathfrak{A}_{ij}=P_{i}\mathfrak{A}P_{j}$ $(i,j=1,2)$.

\section{Main theorem}

In the following we shall prove a part of the main result of this paper.

\begin{theorem}\label{mainthm1} 
Let $\A$ and $\A'$ be two $C^*$-algebras with identities $I_{\A}$ and $I_{\A'}$, respectively, and $P_1$ and $P_2 = I_{\A} - P_1$ nontrivial symmetric projections in $\A$. Suppose that $\A$ satisfies
\begin{eqnarray*}
  &&\left(\spadesuit\right) \ \ \  \ \ \  P_j\A X = \left\{0\right\} \ \ \  \mbox{implies} \ \ \ X = 0.
	%	\\&& \left(\clubsuit \right) \ \ \ \mbox{If} \ B \in \A \mbox{such that} \ q_{n*}(B, \A, ..., \A) = \left\{0\right\} \  \mbox{implies} \ B \in i\mathbb{R}I_{\A},
\end{eqnarray*}
Even more, suppose that $\varphi: \A \rightarrow \A'$ is a bijective unital map which satisfies
\begin{eqnarray*}
% &&\left(\clubsuit\right) \ \ \  \ \ \  \varphi(P_j)\A' Y = \left\{0\right\} \ \ \  \mbox{implies} \ \ \ Y = 0\\
% and && \\
 &&\left(\bullet\right)\varphi(p_{n_*}(A,B,\Xi,...,\Xi)) = p_{n_*}(\varphi(A),\varphi(B),\varphi(\Xi),...,\varphi(\Xi)),
\end{eqnarray*}
 for all $A, B \in \A$ and $\Xi \in \left\{P_1, P_2,I_{\A}\right\}$. Then $\varphi$ is $*$-additive.
\end{theorem}

The following claims and lemmas have the same hypotheses as the Theorem \ref{mainthm1} and
we need them to prove the $*$-additivity of $\varphi$. 
\begin{claim}\label{claim1}
$*(\A_{jk})\subset \A_{kj}$, for $j,k\in \{1,2\}$.
\end{claim}
\begin{proof}
It follows from the symmetry of $P_1$ and $P_2$.
\end{proof}

The next result can be found in \cite{Ferco2} but for brevity we put the proof here
\begin{claim}\label{claim2}
Let $X,Y$ and $H$ be in $\A$ such that $\varphi(H) = \varphi(X) + \varphi(Y)$. Then,
 given $Z \in \A$,
$$\varphi(p_{n_*}(H,Z,\Xi,...,\Xi)) = \varphi(p_{n_*}(X,Z,\Xi,...,\Xi))
                                                   + \varphi(p_{n_*}(Y,Z,\Xi,...,\Xi))$$
and
$$
\varphi(p_{n_*}(Z,H,\Xi,...,\Xi)) = \varphi(p_{n_*}(Z,X,\Xi,...,\Xi)) 
                                                   + \varphi(p_{n_*}(Z,Y,\Xi,...\Xi))$$
for $\Xi \in \left\{P_1, P_2,I_{\A}\right\}$.

%$$
%\5begin{aligned}
%\varphi(q_{n*}(T_1,...,T_{i-1},H,T_{i+1},...,T_n)) &= \varphi(q_{n*}(T_1,...,T_{i-1},X,T_{i+1},...,T_n)) \\
%                                                   &+ \varphi(q_{n*}(T_1,...,T_{i-1},Y,T_{i+1},...,T_n))
%\end{aligned}
%$$
%for $T_j = I_\A$ or $T_j = P$ with $j = 1,...,n$.
\end{claim}
\begin{proof}
Using the definition of $\varphi$ and multilinearity of $p_{n_*}$ we obtain
$$
\begin{aligned}
\varphi(p_{n_*}(H,Z,\Xi,...,\Xi)) &= p_{n_*}(\varphi(H),\varphi(Z),\varphi(\Xi),...,\varphi(\Xi)) \\
                             &= p_{n_*}(\varphi(X)+\varphi(Y),\varphi(Z),\varphi(\Xi),...,\varphi(\Xi)) \\
                             &= p_{n_*}(\varphi(X),\varphi(Z),\varphi(\Xi),...,\varphi(\Xi)) \\
                             &+ p_{n_*}(\varphi(Y),\varphi(Z),\varphi(\Xi),...,\varphi(\Xi)) \\
                             &= \varphi(p_{n_*}(X,Z,\Xi,...,\Xi)) \\
                             &+ \varphi(p_{n_*}(X,Z,\Xi,...,\Xi)).
\end{aligned}
$$
In a similar way we have

$
\varphi(p_{n_*}(Z,H,\Xi,...,\Xi)) = \varphi(p_{n_*}(Z,X,\Xi,...,\Xi)) 
                                                   + \varphi(p_{n_*}(Z,Y,\Xi,...,\Xi)).
$
\end{proof}

\begin{claim}\label{claim3}  $\varphi(0) = 0$.
\end{claim}
\begin{proof}
Since $\varphi$ is surjective, there exists $X \in \A$ such that $\varphi(X) = 0$. Then,
$\varphi(0) = \varphi(p_{n_*}(0,X,I_{\A},...,I_{\A})) = p_{n_*}(\varphi(0),\varphi(X),I_{\A'},...,I_{\A'}) = 0.$
\end{proof}

We follow with a sequence of lemmas that show the additivity of $\varphi$.
\begin{lemma}\label{lema1} For any $A_{11} \in \A_{11}$ and $B_{22} \in \A_{22}$, we have 
$$\varphi(A_{11} + B_{22}) = \varphi(A_{11}) + \varphi(B_{22}).$$
\end{lemma}
\begin{proof}
Since $\varphi$ is surjective, given $\varphi(A_{11})+\varphi(B_{22}) \in \A'$ there exists $H \in \A$ such that 
$\varphi(H) = \varphi(A_{11})+\varphi(B_{22})$, with $H=H_{11}+H_{12}+H_{21}+H_{22}$. Besides, by claims \ref{claim2} and \ref{claim3}
$$
\varphi(p_{n_*}(P_1,H,P_1,...,P_1)) = \varphi(p_{n_*}(P_1,A_{11},P_1,...,P_1)) + \varphi(p_{n_*}(P_1,B_{22},P_1,...,P_1)), 
$$
that is, 
$$\varphi(-H_{21} + H_{21}^*)= \varphi(0) + \varphi(0) = 0.$$
Then, by injectivity of $\varphi$, $-H_{21} + H_{21}^* = 0$. Thus $H_{21} = 0$. Moreover, 
$$
\varphi(p_{n_*}(P_2,H,P_2,...,P_2)) = \varphi(p_{n_*}(P_2,A_{11},P_2,...,P_2)) + \varphi(p_{n_*}(P_2,B_{22},P_2,...,P_2)), 
$$
that is,
$$\varphi(-H_{12} + H_{12}^*)= 0.$$
Again, by injectivity of $\varphi$ we conclude that $H_{12} = 0$.

Furthermore, given $D_{21}\in \A_{21}$,
$$ 
\begin{aligned}
\varphi(p_{n_*}(D_{21},H,P_1,...,P_1)) &=\varphi(p_{n_*}(D_{21},A_{11},P_1,...,P_1)) \\
                                       &+\varphi(p_{n_*}(D_{21},B_{22},P_1,...,P_1)), 
\end{aligned}																			
$$
that is,
$$
\varphi(D_{21}H_{11} - (D_{21}H_{11})^*) = \varphi(D_{21}A_{11} - (D_{21}A_{11})^*).
$$
Then we conclude, by injectivity of $\varphi$, that $D_{21}H_{11} - (D_{21}H_{11})^* = D_{21}A_{11} - (D_{21}A_{11})^*$, that is, 
$D_{21}(H_{11} - A_{11}) = 0$. Even more, $P_2\A(H_{11} - A_{11}) = 0$, which implies that $H_{11} = A_{11}$ by $\left(\spadesuit\right)$.

Finally, given $D_{12}\in \A_{12}$, a similar calculation gives us $H_{22} = B_{22}$. Therefore $H = A_{11} + B_{22}$.
\end{proof}

\begin{lemma}\label{lema2}
For any $A_{12} \in \A_{12}$ and $B_{21} \in \A_{21}$, we have $\varphi(A_{12} + B_{21}) = \varphi(A_{12}) + \varphi(B_{21})$.
\end{lemma}

\begin{proof}
Since $\varphi$ is surjective, given $\varphi(A_{12})+\varphi(B_{21}) \in \A'$ there exists $H \in \A$ such that 
$\varphi(H) = \varphi(A_{12})+\varphi(B_{21})$, with $H=H_{11}+H_{12}+H_{21}+H_{22}$. Now, by Claims \ref{claim2} and \ref{claim3}
$$
\varphi(p_{n_*}(P_1,H,P_1,...,P_1)) = \varphi(p_{n_*}(P_1,A_{12},P_1,...,P_1)) + \varphi(p_{n_*}(P_1,B_{21},P_1,...,P_1)),
$$
that is,
$$
\varphi(-H_{21} + H_{21}^*) = \varphi(-B_{21} + B_{21}^*).
$$
Then, by injectivity of $\varphi$, $-H_{21} + H_{21}^* = -B_{21} + B_{21}^*$. Thus $H_{21} = B_{21}$. Moreover,
$$
\varphi(p_{n_*}(P_2,H,P_2,...,P_2)) = \varphi(p_{n_*}(P_2,A_{12},P_2,...,P_2)) + \varphi(p_{n_*}(P_2,B_{21},P_2,...,P_2)),$$
that is,
$$
\varphi(-H_{12} + H_{12}^*) = \varphi(-A_{12} + A_{12}^*).
$$
Again, by injectivity of $\varphi$ we conclude that $H_{12} = A_{12}$.

Furthermore, given $D_{21}\in \A_{21}$,
$$ 
\begin{aligned}
\varphi(D_{21}H_{11} - (D_{21}H_{11})^*) &= \varphi(p_{n_*}(D_{21},H,P_1,...,P_1)) \\
                                         &= \varphi(p_{n_*}(D_{21},A_{12},P_1,...,P_1)) \\
																				 &+ \varphi(p_{n_*}(D_{21},B_{21},P_1,...,P_1)) = 0.
\end{aligned}
$$
Then we conclude, by injectivity of $\varphi$, that $D_{21}H_{11} - (D_{21}H_{11})^* = 0$, that is, $D_{21}H_{11} = 0$. Even more, $P_2\A H_{11} = 0$, 
which implies that $H_{11} = 0$ by $\left(\spadesuit\right)$.

Finally, given $D_{12}\in \A_{12}$, a similar calculation gives us $H_{22} = 0$. Therefore, we conclude that $H = A_{12} + B_{21}$.
\end{proof}

\begin{lemma}\label{lema3}
For any $A_{11} \in \A_{11}$, $B_{12} \in \A_{12}$, $C_{21} \in \A_{21}$ and $D_{22} \in \A_{22}$ we have 
$$\varphi(A_{11} + B_{12} + C_{21} + D_{22}) = \varphi(A_{11}) + \varphi(B_{12}) + \varphi(C_{21}) + \varphi(D_{22}).$$
\end{lemma}
\begin{proof}
Since $\varphi$ is surjective, given $\varphi(A_{11})+\varphi(B_{12})+\varphi(C_{21}) + \varphi(D_{22})\in \A'$ there exists $H \in \A$ such that 
$\varphi(H) = \varphi(A_{11})+\varphi(B_{12})+\varphi(C_{21}) + \varphi(D_{22})$, with $H=H_{11}+H_{12}+H_{21}+H_{22}$. Even more, by Lemmas \ref{lema1} and \ref{lema2}
$$
\varphi(H) = \varphi(A_{11})+\varphi(B_{12})+\varphi(C_{21}) + \varphi(D_{22}) = \varphi(A_{11} + D_{22})+\varphi(B_{12} + C_{21}).
$$
Now, observing that
$p_{n_*}(P_1,A_{11} + D_{22},...,P_1) = 0 = p_{n_*}(P_1,B_{12},P_1,...,P_1)$ and by Claims \ref{claim2} and \ref{claim3} we obtain
$$
\begin{aligned}
&\varphi(p_{n_*}(P_1,H,P_1,...,P_1)) \\
&= \varphi(p_{n_*}(P_1,A_{11} + D_{22},...,P_1)) + \varphi(p_{n_*}(P_1,B_{12} + C_{21},P_1,...,P_1)) \\
&= \varphi(p_{n_*}(P_1,C_{21},P_1,...,P_1)), 
\end{aligned}
$$
that is, 
$$
\varphi(-H_{21} + H_{21}^*) = \varphi(-C_{21} + C_{21}^*).
$$
Then, by injectivity of $\varphi$, $-H_{21} + H_{21}^* = -C_{21} + C_{21}^*$. Thus $H_{21} = C_{21}$.

In a similar way, using $P_2$ rather than $P_1$ in the previous calculation, we conclude that $H_{12} = B_{12}$. Also, given $X_{21} \in \A_{21}$,
$$
\begin{aligned}
&\varphi(p_{n_*}(X_{21},H,P_1,...,P_1)) \\
&= \varphi(p_{n_*}(X_{21},A_{11} + D_{22},...,P_1)) + \varphi(p_{n_*}(X_{21},B_{12} + C_{21},P_1,...,P_1)) \\
&= \varphi(p_{n_*}(X_{21},A_{11},P_1,...,P_1)), 
\end{aligned}
$$
since $p_{n_*}(X_{21},B_{12} + C_{21},P_1,...,P_1) = 0 = p_{n_*}(X_{21},D_{22},...,P_1)$. 
Again, by injectivity of $\varphi$ we conclude, by following the same strategy as in the proof of Lemma \ref{lema1}, that $H_{11} = A_{11}$. 
Now, using $P_2$ rather than $P_1$ and $X_{12}$ rather than $X_{21}$ in the previous calculation we obtain $H_{22} = D_{22}$. 
Therefore, $H = A_{11} + B_{12} + C_{21} + D_{22}$. 
\end{proof}

\begin{lemma}\label{lema4}
For all $A_{jk}, B_{jk} \in \A_{jk}$ we have $\varphi(A_{jk} + B_{jk}) = \varphi(A_{jk}) + \varphi(B_{jk})$ for $j \neq k$.
\end{lemma}
\begin{proof}
We shall prove the case $j=1$ and $k = 2$, since the other case is done in a similar way. Since $\varphi$ is surjective, given 
$\varphi(A_{12}) + \varphi(B_{12})\in \A'$ and $\varphi(-A_{12}^*) + \varphi(-B_{12}^*)$ there exist $H \in \A$ and $T \in \A$ such that 
$\varphi(H) = \varphi(A_{12}) + \varphi(B_{12})$ and $\varphi(T) = \varphi(-A_{12}^*) + \varphi(-B_{12}^*)$, with $H=H_{11}+H_{12}+H_{21}+H_{22}$
and $T = T_{11} + T_{12} + T_{21} + T_{22}$.

Firstly, we show that $H\in \A_{12}$: by Claim \ref{claim2}
$$
\begin{aligned}
\varphi(-H_{21} + H_{21}^*) &= \varphi(p_{n_*}(P_1,H,P_1,...,P_1)) \\
                            &= \varphi(p_{n_*}(P_1,A_{12},P_1,...,P_1)) + \varphi(p_{n_*}(P_1,B_{12},P_1,...,P_1)) = 0.
\end{aligned}
$$
Then, by injectivity of $\varphi$ we obtain $H_{21} = 0$. Also, given $D_{12}\in \A_{12}$,
$$
\begin{aligned}
\varphi(D_{12}H_{22} - (D_{12}H_{22})^*) &= \varphi(p_{n_*}(D_{12},H,P_2,...,P_2)) \\
                                         &= \varphi(p_{n_*}(D_{12},A_{12},P_2,...,P_2)) \\
																				 &+ \varphi(p_{n_*}(D_{12},B_{12},P_2,...,P_2)) = 0,
\end{aligned}
$$
that is, $D_{12}H_{22} = 0$, which implies that $H_{22} = 0$ by $\left(\spadesuit\right)$. Now, using $D_{21}\in \A_{21}$ rather than $D_{12}$ in the previous calculation, 
we conclude that $H_{11} = 0$. Therefore, $H = H_{12}\in \A_{12}$.

In a similar way, we obtain $T = T_{21}\in A_{21}$.

Finally, by Lemma \ref{lema3}

$$
\begin{aligned}
\varphi(A_{12} + B_{12} - A_{12}^* - B_{12}^*) &= \varphi(p_{n_*}(P_1 + A_{12},P_2 + B_{12},P_2,...,P_2)) \\
                                               &= p_{n_*}(\varphi(P_1 + A_{12}),\varphi(P_2 + B_{12}),\varphi(P_2),...,\varphi(P_2)) \\
																							 &= p_{n_*}(\varphi(P_1),\varphi(P_2),\varphi(P_2),...,\varphi(P_2)) \\
																							 &+ p_{n_*}(\varphi(P_1),\varphi(B_{12}),\varphi(P_2),...,\varphi(P_2)) \\
																							 &+ p_{n_*}(\varphi(A_{12}),\varphi(P_2),\varphi(P_2),...,\varphi(P_2)) \\
																							 &+ p_{n_*}(\varphi(A_{12}),\varphi(B_{12}),\varphi(P_2),...,\varphi(P_2)) \\
																							 &= \varphi(p_{n_*}(P_1,P_2,P_2,...,P_2)) \\
																							 &+ \varphi(p_{n_*}(P_1,B_{12},P_2,...,P_2)) \\
																							 &+ \varphi(p_{n_*}(A_{12},P_2,P_2,...,P_2)) \\
																							 &+ \varphi(p_{n_*}(A_{12},B_{12},P_2,...,P_2)) \\
																							 &= \varphi(A_{12} - A_{12}^*) + \varphi(B_{12} - B_{12}^*) \\
																							 &= \varphi(A_{12}) + \varphi(B_{12}) + \varphi(-A_{12}^*) + \varphi(-B_{12}^*) \\
																							 &= \varphi(H_{12}) + \varphi(T_{21}) = \varphi(H_{12} + T_{21}).																						
\end{aligned}
$$

Since $\varphi$ is injective, we have $A_{12} + B_{12} - A_{12}^* - B_{12}^* = H_{12} + T_{21}$, i.e., $H = H_{12} = A_{12} + B_{12}$.
\end{proof}
\begin{lemma}\label{lema5}
For all $A_{jj}, B_{jj} \in \A_{jj}$, we have $\varphi(A_{jj} + B_{jj}) = \varphi(A_{jj}) + \varphi(B_{jj})$ for $j \in \left\{1,2\right\}.$
\end{lemma}
\begin{proof}
We shall prove the case $j=1$, since the other case is done in a similar way. Since $\varphi$ is surjective, given 
$\varphi(A_{11}) + \varphi(B_{11})\in \A'$ there exists $H \in \A$ such that 
$\varphi(H) = \varphi(A_{11}) + \varphi(B_{11})$, with $H=H_{11}+H_{12}+H_{21}+H_{22}$. Now, by claim \ref{claim2}
$$
\begin{aligned}
\varphi(-H_{21} + H_{21}^*) &= \varphi(p_{n_*}(P_1,H,P_1,...,P_1)) \\
                            &= \varphi(p_{n_*}(P_1,A_{11},P_1,...,P_1)) + \varphi(p_{n_*}(P_1,B_{11},P_1,...,P_1)) = 0.
\end{aligned}
$$
Then, by injectivity of $\varphi$ we obtain $H_{21} = 0$. Also,
$$
\begin{aligned}
\varphi(-H_{12} + H_{12}^*) &= \varphi(p_{n_*}(P_2,H,P_2,...,P_2)) \\
                            &= \varphi(p_{n_*}(P_2,A_{11},P_2,...,P_2)) + \varphi(p_{n_*}(P_2,B_{11},P_2,...,P_2)) = 0,
\end{aligned}
$$
that is, $H_{12} = 0$ by injectivity of $\varphi$.  Moreover, given $D_{12}\in \A_{12}$,
$$
\begin{aligned}
\varphi(D_{12}H_{22} - (D_{12}H_{22})^*) &= \varphi(p_{n_*}(D_{12},H,P_2,...,P_2)) \\
																				 &= \varphi(p_{n_*}(D_{12},A_{11},P_2,...,P_2)) \\
																				 &+ \varphi(p_{n_*}(D_{12},B_{11},P_2,...,P_2)) = 0.
\end{aligned}
$$ 
Then, by injectivity of $\varphi$, $D_{12}H_{22}=0$, which implies that $H_{22} = 0$ by $\left(\spadesuit\right)$. Finally, given $D_{21}\in \A_{21}$, 
by Lemmas \ref{lema3} and \ref{lema4} we have
$$
\begin{aligned}
\varphi(D_{21}H_{11} - (D_{21}H_{11})^*) &= \varphi(p_{n_*}(D_{21},H,P_1,...,P_1)) \\
																				 &= \varphi(p_{n_*}(D_{21},A_{11},P_1,...,P_1)) \\
																				 &+ \varphi(p_{n_*}(D_{21},B_{11},P_1,...,P_1)) \\
																				 &= \varphi(D_{21}A_{11} - (D_{21}A_{11})^*) \\
																				 &+ \varphi(D_{21}B_{11} - (D_{21}B_{11})^*) \\
																				 &= \varphi(D_{21}A_{11}) + \varphi(-(D_{21}A_{11})^*) \\
																				 &+ \varphi(D_{21}B_{11}) + \varphi(-(D_{21}B_{11})^*) \\
																				 &= \varphi(D_{21}A_{11} + D_{21}B_{11}) \\
																				 &+ \varphi(-(D_{21}A_{11})^*-(D_{21}B_{11})^*) \\
																				 &= \varphi(D_{21}(A_{11} + B_{11}) - (A_{11}^* + B_{11}^*)D_{21}^*),
\end{aligned}
$$ 
that is, $D_{21}H_{11} - (D_{21}H_{11})^* = D_{21}(A_{11} + B_{11}) - (A_{11}^* + B_{11}^*)D_{21}^*$, by injectivity of $\varphi$. Thus, 
$D_{21}(H_{11} - (A_{11} + B_{11})) = 0$, which implies that $H_{11} = A_{11} + B_{11}$ by $\left(\spadesuit\right)$.
\end{proof}

Now we are able to show that $\varphi$ is $*$-additive. Using Lemmas \ref{lema3}, \ref{lema4} and \ref{lema5} we have, for all $A,B \in \A$,
$$
\begin{aligned}
\varphi(A + B) &= \varphi(A_{11}+A_{12}+A_{21}+A_{22}+B_{11}+B_{12}+B_{21}+B_{22}) \\
               &= \varphi(A_{11}+B_{11})+\varphi(A_{12}+B_{12})+\varphi(A_{21}+B_{21})+\varphi(A_{22}+B_{22}) \\
               &= \varphi(A_{11})+\varphi(B_{11})+\varphi(A_{12})+\varphi(B_{12}) \\
							 &+\varphi(A_{21})+\varphi(B_{21})+\varphi(A_{22})+\varphi(B_{22}) \\
               &= \varphi(A_{11}+A_{12}+A_{21}+A_{22}) \\
							 &+ \varphi(B_{11}+B_{12}+B_{21}+B_{22}) = \varphi(A) + \varphi(B).
\end{aligned}
$$

Besides, given $A \in \A$, by additivity of $\varphi$
$$
\begin{aligned}
2^{n-2}(\varphi(A) - \varphi(A)^*) &= p_{n_*}(\varphi(A), I_{\A'},...,I_{\A'}) = \varphi(p_{n_*}(A,I_{\A},...,I_{\A})) \\
                          &= \varphi(2^{n-2}(A - A^*)) = 2^{n-2}\varphi(A - A^*) = 2^{n-2}(\varphi(A) - \varphi(A^*))
\end{aligned}
$$
and then we conclude that $\varphi(A^*) = \varphi(A)^*$. This completes the proof of Theorem \ref{mainthm1}. 

\begin{theorem}\label{mainthm2} 
Let $\A$ and $\A'$ be two $C^*$-algebras with identities $I_{\A}$ and $I_{\A'}$, respectively, and $P_1$ and $P_2 = I_{\A} - P_1$ nontrivial symmetric projections in $\A$. 
Suppose that $\A$ satisfies
\begin{eqnarray*}
  &&\left(\spadesuit\right) \ \ \  \ \ \  P_j\A X = \left\{0\right\} \ \ \  \mbox{implies} \ \ \ X = 0.	
\end{eqnarray*}
Even more, suppose that $\varphi: \A \rightarrow \A'$ is a complex scalar multiplication bijective unital map which satisfies
\begin{eqnarray*}
 &&\left(\clubsuit\right) \ \ \  \ \ \  \varphi(P_j)\A' Y = \left\{0\right\} \ \ \  \mbox{implies} \ \ \ Y = 0\\
 and && \\
 &&\left(\bullet\right)\varphi(p_{n_*}(A,B,\Xi,...,\Xi)) = p_{n_*}(\varphi(A),\varphi(B),\varphi(\Xi),...,\varphi(\Xi)),
\end{eqnarray*}
 for all $A, B \in \A$ and $\Xi \in \left\{P_1, P_2,I_{\A}\right\}$. Then $\varphi$ is $*$-isomorphism.
\end{theorem}

With this hypothesis, we have already proved that $\varphi$ is $*$-additive. It remains for us to show that $\varphi$ is multiplicative. 
In order to do that we will prove some more lemmas. Firstly, we observe that, for any $A \in \A$,
\begin{remark}
Denoting by $\mathcal{R}(A)$ and $\mathcal{I}(A)$ the real part and imaginary part of $A \in \A$, respectively, we have
$$
\begin{aligned}
\varphi(2^{n-1}\mathcal{R}(A)) &= \varphi(2^{n-2}(A + A^*)) = \varphi(p_{n_*}(iA,-iI_{\A},I_{\A},...,I_{\A})) \\
                               &= p_{n_*}(i\varphi(A),-i\varphi(I_{\A}),I_{\A'},...,I_{\A'}) = 2^{n-2}(\varphi(A) + \varphi(A)^*) \\
															 &= 2^{n-1}\mathcal{R}(\varphi(A))
\end{aligned}															
$$
\noindent and
$$
\begin{aligned}
\varphi(2^{n-1}i\, \mathcal{I}(A)) &= \varphi(2^{n-2}(A - A^*)) = \varphi(p_{n_*}(A,I_{\A},...,I_{\A})) \\
                                    &= p_{n_*}(\varphi(A),I_{\A'},...,I_{\A'}) = 2^{n-2}(\varphi(A) - \varphi(A)^*) \\
															      &= 2^{n-1}i\,\mathcal{I}(\varphi(A)).
\end{aligned}
$$ 
\end{remark}

Even more,
\begin{claim}\label{claim4}
$Q_j = \varphi(P_j)$ is a symmetric projection in $\A'$, with $j\in \{1,2\}$.
\end{claim}
\begin{proof}
Since $\varphi$ is a complex scalar multiplication, it follows that
$$
\begin{aligned}
2^{n-1}i\,Q_j &= 2^{n-1}i\varphi(P_j) = \varphi(2^{n-1}iP_j) = \varphi(p_{n_*}(iP_j,P_j,I_{\A},...,I_{\A})) \\
            &= p_{n_*}(i\varphi(P_j),\varphi(P_j),I_{\A'},...,I_{\A'})) = 2^{n-1}i\varphi(P_j)^2 = 2^{n-1}i\,{Q_j}^2.                                                          
\end{aligned}
$$
Then we can conclude that $Q_j = {Q_j}^2$. Moreover, since $P_j$ is a symmetric projection in $\A$ we have that $p_{n_*}(P_j,I_{\A},...,I_{\A}) = 0$. Besides,
$$
0 = \varphi(0) = \varphi(p_{n_*}(P_j,I_{\A},...,I_{\A})) = p_{n_*}(Q_j,I_{\A'},...,I_{\A'}). 
$$
Thus, $Q_j - {Q_j}^*=0$, that is, $Q_j = {Q_j}^*$.
\end{proof}

\begin{lemma}\label{lemanovo}
For all $A\in \A$, $\varphi(P_jA) = \varphi(P_j)\varphi(A)$ and $\varphi(AP_j) = \varphi(A)\varphi(P_j)$.
\end{lemma}
\begin{proof}
Firstly, observe that
$$
p_{n_*}(iA,P_j,I_{\A},...,I_{\A}) = 2^{n-2}i(AP_j + P_jA^*)
$$
and
$$p_{n_*}(A,P_j,I_{\A},...,I_{\A}) = 2^{n-2}(AP_j -P_jA^*).$$
Still, by $(\bullet)$ and $*$-additivity of $\varphi$,
$$
\begin{aligned}
\varphi(2^{n-2}i(AP_j + P_jA^*)) &= \varphi(p_{n_*}(iA,P_j,I_{\A},...,I_{\A})) \\
                                 &= p_{n_*}(\varphi(iA),\varphi(P_j),I_{\A'},...,I_{\A'}) \\
                                 &= 2^{n-2}i(\varphi(A)\varphi(P_j) + \varphi(P_j)\varphi(A)^*)
\end{aligned}
$$
and
$$
\begin{aligned}
\varphi(2^{n-2}(AP_j - P_jA^*)) &= \varphi(p_{n_*}(A,P_j,I_{\A},...,I_{\A})) \\
                                 &= p_{n_*}(\varphi(A),\varphi(P_j),I_{\A'},...,I_{\A'}) \\
                                 &= 2^{n-2}(\varphi(A)\varphi(P_j) - \varphi(P_j)\varphi(A)^*).
\end{aligned}
$$
Now, since $\varphi$ is $*$-additive, multiplying the second equality by $i$ and adding these two equations we obtain $\varphi(AP_j) = \varphi(A)\varphi(P_j)$.
The second statement is obtained in a similar way.
\end{proof}
\begin{lemma}\label{lemadireto}
$\varphi(\A_{jk}) \subset \A_{jk}'$.
\end{lemma}
\begin{proof}
Given $X\in \A_{jk}$, we have $X = P_jXP_k$ and then, by Lemma \ref{lemanovo},\linebreak $\varphi(X) = \varphi(P_j)\varphi(XP_k) = \varphi(P_j)\varphi(X)\varphi(P_k) \in \A_{jk}'$.
\end{proof}

\begin{lemma}\label{lemam2}
To $j\neq k:$
\begin{itemize}
\item If $A_{jk} \in \A_{jk}$ and $B_{kk} \in \A_{kk}$ then $\varphi(A_{jk}B_{kk})=\varphi(A_{jk})\varphi(B_{kk})$;
\item If $A_{jj} \in \A_{jj}$ and $B_{jk} \in \A_{jk}$ then $\varphi(A_{jj}B_{jk})=\varphi(A_{jj})\varphi(B_{jk})$;
\item If $A_{jk} \in \A_{jk}$ and $B_{kj} \in \A_{kj}$ then $\varphi(A_{jk}B_{kj})=\varphi(A_{jk})\varphi(B_{kj})$.
\end{itemize}
\end{lemma}
\begin{proof}
In order to prove the first statement, on the one hand, by Lemma \ref{lemadireto}
$$
\begin{aligned}
\varphi(A_{jk}B_{kk}) - \varphi(A_{jk}B_{kk})^* &= \varphi(A_{jk}B_{kk} - (A_{jk}B_{kk})^*) \\
                                                &= \varphi(p_{n_*}(A_{jk},B_{kk},P_k,...,P_k)) \\
                                                &= p_{n_*}(\varphi(A_{jk}), \varphi(B_{kk}), Q_k,...,Q_k)\\
                                                &= \varphi(A_{jk})\varphi(B_{kk}) - (\varphi(A_{jk})\varphi(B_{kk}))^*
\end{aligned}
$$
and then $\mathcal{I}(\varphi(A_{jk}B_{kk})) = \mathcal{I}(\varphi(A_{jk})\varphi(B_{kk}))$. On the other hand, using $iA_{jk}$ rather than $A_{jk}$ we obtain
$\mathcal{R}(\varphi(A_{jk}B_{kk})) = \mathcal{R}(\varphi(A_{jk})\varphi(B_{kk}))$. It concludes that $\varphi(A_{jk}B_{kk}) = \varphi(A_{jk})\varphi(B_{kk})$.

The others statements are proved in a similar way.
\end{proof}
\begin{lemma}\label{lemam3}
If $A_{jj},B_{jj} \in \A_{jj}$ then $\varphi(A_{jj}B_{jj}) = \varphi(A_{jj})\varphi(B_{jj})$.
\end{lemma}
\begin{proof}
Let $X_{kj}$ be an element of $\A_{kj}$, with $j \neq k$. Using Lemma \ref{lemam2} we obtain
$$
\varphi(X_{kj})\varphi(A_{jj}B_{jj}) = \varphi(X_{kj}A_{jj}B_{jj}) = \varphi(X_{kj}A_{jj})\varphi(B_{jj}) = \varphi(X_{kj})\varphi(A_{jj})\varphi(B_{jj}),
$$
that is,
$$
\varphi(X_{kj})(\varphi(A_{jj}B_{jj}) - \varphi(A_{jj})\varphi(B_{jj})) = 0.
$$
Now, by Lemma \ref{lemadireto}, $\varphi(X_{kj}) \in \A'_{kj}$ as well as $\varphi(A_{jj}B_{jj})$ and $\varphi(A_{jj})\varphi(B_{jj}) \in \A'_{jj}$.
Then, $Q_k\A'(\varphi(A_{jj}B_{jj}) - \varphi(A_{jj})\varphi(B_{jj})) = 0$, which implies that $\varphi(A_{jj}B_{jj}) = \varphi(A_{jj})\varphi(B_{jj})$ by $\left(\clubsuit\right)$.
\end{proof}
By additivity of $\varphi$ and Lemmas \ref{lemam2} and \ref{lemam3}, it follows, for all $A$, $B \in \A$, that $\varphi(AB) = \varphi(A)\varphi(B)$.
It concludes the proof of Theorem \ref{mainthm2}.

\section{Corollaries} 

Let us present some consequences of our main result. The first one provides the conjecture that appears in \cite{Ferco} to the case of multiplicative $*$-Lie-type maps:

\begin{corollary}   
Let $\A$ and $\A'$ be two $C^*$-algebras with identities $I_\A$ and $I_{\A'}$, respectively, and $P_1$ and $P_2 = I_{\A} - P_1$ nontrivial symmetric projections in $\A$. Suppose that $\A$ satisfies
\begin{eqnarray*}
  &&\left(\spadesuit\right) \ \ \  \ \ \  P_j\A X = \left\{0\right\} \ \ \  \mbox{implies} \ \ \ X = 0.	
\end{eqnarray*}
Even more, suppose that $\varphi: \A \rightarrow \A'$ is a complex scalar multiplication bijective unital map which satisfies
\begin{eqnarray*}
 &&\left(\clubsuit\right) \ \ \  \ \ \  \varphi(P_j)\A' Y = \left\{0\right\} \ \ \  \mbox{implies} \ \ \ Y = 0.
\end{eqnarray*}
Then $\varphi: \A \rightarrow \A'$ is a multiplicative $*$-Lie $n$-map if and only if $\varphi$ is a
$*$-isomorphism.
\end{corollary}

Observing that prime $C^*$-algebras satisfy $(\spadesuit), (\clubsuit)$ we have the following result: 

\begin{corollary} 
Let $\A$ and $\A'$ be prime $C^*$-algebras with identities $I_{\A}$ and $I_{\A'}$, respectively, and $P_1$ and $P_2 = I_{\A} - P_1$ nontrivial projections in $\A$. Then a complex scalar multiplication $\varphi: \A \rightarrow \A'$ is a bijective unital multiplicative $*$-Lie $n$-map if and only if $\varphi$ is a $*$-isomorphism.
\end{corollary}

A von Neumann algebra $\mathcal{M}$ is a weakly closed, self-adjoint algebra of operators on a Hilbert space $\mathcal{H}$ containing the identity operator $I$. As an application on von Neumann algebras we have the following: 

\begin{corollary}
Let $\mathcal{M}$ be a von Neumann algebra without central summands of type $I_1$. Then a complex scalar multiplication $\varphi: \mathcal{M} \rightarrow \mathcal{M}$ is a bijective unital multiplicative $*$-Lie $n$-map if and only if $\varphi$ is a $*$-isomorphism.
\end{corollary}
\begin{proof}
Let $\mathcal{M}$ be a von Neumann algebra. It is shown in \cite{Bai} and \cite{Mie} that if a von Neumann algebra has no central summands of type $I_1$, then $\mathcal{M}$ satisfies the following assumption: 
\begin{itemize}
\item $P_j\mathcal{M}X = \left\{0\right\} \Rightarrow X = 0$.
\end{itemize}
Thus, by Theorem \ref{mainthm2} the corollary is true.
\end{proof}

 To finish, $\mathcal{M}$ is a factor von Neumann algebra if its center only contains the scalar operators. It is well known that a factor von Neumann algebra is prime and then we have the following:

\begin{corollary}
Let $\mathcal{M}$ be a factor von Neumann algebra. Then a complex scalar multiplication $\varphi: \mathcal{M} \rightarrow \mathcal{M}$ is a bijective unital multiplicative $*$-Lie $n$-map if and only if $\varphi$ is a $*$-isomorphism.
\end{corollary}

\end{document}